\documentclass[12pt]{amsart}
\usepackage{amsfonts}

\usepackage{amsmath}
\usepackage{graphicx}

\newtheorem{theorem}{Theorem}

\newtheorem{corollary}[theorem]{Corollary}

\newtheorem{example}[theorem]{Example}

\newtheorem{proposition}[theorem]{Proposition}

\begin{document}
\title{Convergence of rational Bernstein operators}
\author[H. Render]{Hermann Render}
\address{School of Mathematical Sciences, University College Dublin, Dublin 4,
Ireland.}
\email{hermann.render@ucd.ie}
\thanks{}

\begin{abstract}
In this paper we discuss convergence properties and error estimates of
rational Bernstein operators introduced by P. Pi\c{t}ul and P.
Sablonni\`{e}re. It is shown that the rational Bernstein operators $R_{n}$
converge to the identity operator if and only if $\Delta _{n},$ the maximal
difference between two consecutive nodes of $R_{n},$ is converging to zero
for $n\rightarrow \infty .$ Error estimates in terms of $\Delta _{n}$ are
provided. Moreover a Voronovskaja theorem is presented which is based on the
explicit computation of higher order moments for the rational Bernstein
operator.
\end{abstract}

\maketitle


\section{Introduction}

Let $C\left[ 0,1\right] $ be the set of all continuous real-valued functions
on the interval $\left[ 0,1\right] $. The classical Bernstein operator $%
B_{n}:C\left[ 0,1\right] \rightarrow C\left[ 0,1\right] \mathcal{\ }$ is
defined by 
\begin{equation*}
B_{n}f\left( x\right) =\sum_{k=0}^{n}f\left( \frac{k}{n}\right) \binom{n}{k}%
x^{k}\left( 1-x\right) ^{n-k},
\end{equation*}
see e.g. \cite{Davis}, \cite{Lore86}. In \cite{PiSa09}, P. Pi\c{t}ul and P.
Sablonni\`{e}re introduced \emph{rational Bernstein operators} which are
positive operators of the form 
\begin{equation}
R_{n}f\left( x\right) :=\sum_{k=0}^{n}f\left( x_{n,k}\right) \overline{w}%
_{n,k}\binom{n}{k}\frac{x^{k}\left( 1-x\right) ^{n-k}}{Q_{n-1}\left(
x\right) }  \label{defRn}
\end{equation}
where $Q_{n-1}\left( x\right) $ is a given strictly positive polynomial over 
$\left[ 0,1\right] $ of degree $\leq n-1.$ Further it is assumed that $%
Q_{n-1}$ has two additional properties: (i) the Bernstein coefficients $%
w_{n-1,k}$ in the representation 
\begin{equation}
Q_{n-1}\left( x\right) =\sum_{k=0}^{n-1}w_{n-1,k-1}\binom{n-1}{k}x^{k}\left(
1-x\right) ^{n-1-k}  \label{eqdefQn}
\end{equation}
are strictly positive and (ii) the sequence $w_{n-1,k},k=0,...,n-1$
satisfies the inequality 
\begin{equation}
\frac{w_{n-1,k-1}w_{n-1,k+1}}{w_{n-1,k}^{2}}<\left( \frac{k+1}{k}\right)
\left( \frac{n-k}{n-k-1}\right)  \tag{W}  \label{eqw}
\end{equation}
for $k=1,...,n-1.$ Then, according to the results in \cite{PiSa09}, there
exist positive weights $\overline{w}_{n,k},k=0,...,n,$ and increasing nodes $%
0=x_{n,0}<x_{n,1}<...<x_{n,n}=1$ such that $R_{n}$ reproduces the constant
function $e_{0}\left( x\right) =1$ and the linear function $e_{1}\left(
x\right) =x,$ i.e. that 
\begin{equation}
R_{n}e_{j}=e_{j}\text{ for }j=0,1.  \label{eqfix}
\end{equation}
The \emph{weights} $\overline{w}_{n,k}$ and the \emph{nodes} $x_{n,k}$ are
uniquely defined through the condition (\ref{eqfix}) and they are given by
the formula 
\begin{eqnarray*}
x_{n,k} &=&\frac{kw_{n,k-1}}{kw_{n,k-1}+\left( n-k\right) w_{n,k}}\text{ for 
}k=1,..,n-1\text{ } \\
\overline{w}_{n,k} &=&\frac{k}{n}w_{n,k-1}+\left( 1-\frac{k}{n}\right)
w_{n,k}\text{ for }k=1,...,n-1
\end{eqnarray*}
and the conditions $x_{n,0}=0$ and $x_{n,n}=1$ and $w_{n,0}=Q_{n-1}\left(
0\right) $ and $w_{n,n}=Q_{n-1}\left( 1\right) .$ It was shown in \cite
{PiSa09} that the rational Bernstein operators $R_{n}$ have the same shape
preserving properties as the classical Bernstein operator $B_{n}.$ Moreover
it was proved that $R_{n}$ converges to the identity operator and that a
Voronovskaja-type theorem holds under the additional assumption that there
exists a positive continuous function $\varphi $ such that 
\begin{equation*}
w_{n-1,k}=\varphi \left( \frac{k}{n-1}\right) \binom{n-1}{k}\text{ for }%
k=0,...,n-1
\end{equation*}
for all natural numbers $n$. The main purpose of this article is to study
the convergence of the rational Bernstein operators in the general case. Our
main result states that the operators $R_{n}$ converge to the identity
operator if and only if 
\begin{equation}
\Delta _{n}=\sup_{k=0,..,n-1}\left| x_{n,k+1}-x_{n,k}\right|
\label{eqDeltan}
\end{equation}
converges to $0.$ The main innovation in the present article is the
computation and estimation of the moments 
\begin{equation*}
R_{n}\left( e_{1}-x\right) ^{r}\left( x\right) \text{ and }R_{n}\left(
e_{r}\right) \left( x\right) -x^{r}
\end{equation*}
for the rational Bernstein operator $R_{n}$ where $e_{r}\left( x\right)
=x^{r}.$ For example, we shall prove the inequality 
\begin{equation*}
\left| R_{n}\left( e_{2}\right) \left( x\right) -x^{2}\right| \leq
\sup_{0\leq k\leq n-1}\left| x_{n,k+1}-x_{n,k}\right| \cdot x\left(
1-x\right)
\end{equation*}
which implies the convergence of $R_{n}$ to the identity operator provided
that $\Delta _{n}\rightarrow 0.$ Convergence results and error estimates of
O. Shisha and B. Mond for positive operators are used for explicit error
estimates. Results of R.G. Mamedov lead to a Voronovskaja theorem for
rational Bernstein operators in the general setting. We shall illustrate the
results by examples which are of different type as those in \cite{PiSa09}.

The paper is organized as follows: in the second section we shall recall
briefly the basic construction of the rational Bernstein operators as given
in \cite{PiSa09}. We shall show that there are many natural examples of
rational Bernstein operators: starting with nodes $%
0=x_{n,0}<x_{n,1}<...<x_{n,n-1}<x_{n,n}=1$ and a positive constant $\gamma
_{n-1,0}>0$ we define 
\begin{equation*}
\gamma _{n-1,k}:=\gamma _{n-1,0}\prod\limits_{l=1}^{k}\frac{1-x_{n,l}}{%
x_{n,l}}
\end{equation*}
for $k=1,...,n$ and $Q_{n-1}\left( x\right) :=\sum_{k=0}^{n-1}\gamma
_{n-1,k}x^{k}\left( 1-x\right) ^{n-1-k}.$ Then $Q_{n-1}$ satisfies property
(W) and 
\begin{equation*}
R_{n}f\left( x\right) =\sum_{k=0}^{n}f\left( x_{n,k}\right) \left( \gamma
_{n-1,k}+\gamma _{n-1,k-1}\right) \frac{x^{k}\left( 1-x\right) ^{n-k}}{%
Q_{n-1}\left( x\right) }
\end{equation*}
is a rational Bernstein operator $R_{n}$ fixing $e_{0}$ and $e_{1}$. In
Section 3 we compute the expressions $R_{n}\left( e_{r}\right) \left(
x\right) -x^{r}$ explicitly and obtain the above-mentioned criterion for the
convergence of $R_{n}.$ Section 4 is devoted to error estimates. In Section
5 we prove a Voronovskaja result. In Section 6 we discuss the special case
of rational Bernstein operators of \cite{PiSa09} and improve some results.
Further we present a sequence of rational Bernstein operators $R_{n}$
converging to the identity operator where the polynomials $Q_{n}\left(
x\right) $ converges pointwise to a \emph{discontinuous} function. In the
final Section 7, we shall comment on links between rational Bernstein
operators and general results about Bernstein operators fixing two functions
in the framework of extended Chebyshev systems.

By $C^{r}\left[ 0,1\right] $ we shall denote the set of all $r$ times
continuously differentiable functions on the unit interval $\left[ 0,1\right]
$ and $\mathbb{N}$ will denote the set of all natural numbers.

\section{Rational Bernstein operators}

For convenience of the reader we recall the basic construction of the
rational Bernstein operator $R_{n}$ as outlined in \cite{PiSa09}. Let $%
Q_{n-1}$ be a polynomial of degree $\leq n-1.$ Instead of the representation
(\ref{eqdefQn}) it is more convenient to work with 
\begin{equation}
Q_{n-1}\left( x\right) =\sum_{k=0}^{n-1}\gamma _{n-1,k}x^{k}\left(
1-x\right) ^{n-1-k}  \label{eqQrep}
\end{equation}
so $\gamma _{n-1,k}=w_{n-1,k}\binom{n-1}{k}.$ Since $x^{k}\left( 1-x\right)
^{n-1-k}=x^{k}\left( 1-x\right) ^{n-k}+x^{k+1}\left( 1-x\right) ^{n-1-k}$ we
infer that 
\begin{equation*}
Q_{n-1}\left( x\right) =\sum_{k=0}^{n}\left( \gamma _{n-1,k}+\gamma
_{n-1,k-1}\right) x^{k}\left( 1-x\right) ^{n-k}
\end{equation*}
with the convention that $\gamma _{n-1,-1}=0$ and $\gamma _{n-1,n}=0.$ In
view of (\ref{defRn}) the requirement $R_{n}1=1$ is then equivalent to 
\begin{equation*}
Q_{n-1}\left( x\right) =\sum_{k=0}^{n}\overline{w}_{n,k}\binom{n}{k}%
x^{k}\left( 1-x\right) ^{n-k}
\end{equation*}
and we conclude that 
\begin{equation*}
\overline{w}_{n,k}\binom{n}{k}=\gamma _{n-1,k}+\gamma _{n-1,k-1}.
\end{equation*}
Further we want that $R_{n}e_{1}=e_{1}$ for the linear function $e_{1}\left(
x\right) =x$ which is equivalent to the identity 
\begin{equation}
xQ_{n-1}\left( x\right) =\sum_{k=0}^{n}x_{n,k}\cdot \overline{w}_{n,k}\binom{%
n}{k}x^{k}\left( 1-x\right) ^{n-k}.  \label{eq1}
\end{equation}
Inserting $x=0$ implies that that $x_{n,0}=0$. From the identity 
\begin{equation*}
xQ_{n-1}\left( x\right) =\sum_{k=0}^{n-1}\gamma _{n-1,k}\cdot x^{k+1}\left(
1-x\right) ^{n-1-k}=\sum_{k=1}^{n}\gamma _{n-1,k-1}x^{k}\left( 1-x\right)
^{n-k}
\end{equation*}
and (\ref{eq1}) we infer that for $k=1,...n$ 
\begin{equation}
x_{n,k}=\frac{\gamma _{n-1,k-1}}{\overline{w}_{n,k}}=\frac{\gamma _{n-1,k-1}%
}{\gamma _{n-1,k}+\gamma _{n-1,k-1}}=\frac{\frac{\gamma _{n-1,k-1}}{\gamma
_{n-1,k}}}{1+\frac{\gamma _{n-1,k-1}}{\gamma _{n-1,k}}}.  \label{eqxnk}
\end{equation}
Hence, given the polynomial $Q_{n-1}\left( x\right) ,$ there is \emph{at
most one} choice for the nodes $x_{n,k}$ and the weights $\overline{w}_{n,k}$
such that $R_{n}$ fixes $e_{0}$ and $e_{1}.$ However, in general the numbers 
$x_{n,k}$ defined by (\ref{eqxnk}) are not in the interval $\left[ 0,1\right]
,$ and they are in general not increasing numbers, for see Example \ref
{ThmExamp} in Section 6. From formula (\ref{eqxnk}) and the fact that $%
f\left( x\right) =x/\left( 1+x\right) $ is strictly increasing we derive
that $x_{n,k}$ is strictly increasing if and only if 
\begin{equation*}
\frac{\gamma _{n-1,k-1}}{\gamma _{n-1,k}}=\frac{w_{n-1,k-1}}{w_{n,k}}\frac{k%
}{n-k}\text{ is strictly increasing.}
\end{equation*}
This is exactly condition (W). The construction of the rational Bernstein
operator $R_{n}$ has the disadvantage that one has to check the condition
(W) for the Bernstein coefficients of the polynomial $Q_{n-1}$ which in
general might be cumbersome.

\begin{example}
Take $Q_{n-1}\left( x\right) =1+x^{2},$ then straightforward calculations
show that 
\begin{eqnarray*}
\gamma _{n-1,k} &=&\binom{n-1}{k}\left( 1+\frac{k\left( k-1\right) }{\left(
n-1\right) \left( n-2\right) }\right)  \\
\gamma _{n-1,k-1}+\gamma _{n-1,k} &=&\binom{n}{k}\frac{n\left( n-1\right)
+k\left( k-1\right) }{n\left( n-1\right) } \\
x_{n,k} &=&\frac{k}{n-2}\frac{\left( n-1\right) \left( n-2\right) +\left(
k-1\right) \left( k-2\right) }{n\left( n-1\right) +k\left( k-1\right) }
\end{eqnarray*}
and the rational Bernstein operator $R_{n}$ is given by 
\begin{equation*}
R_{n}f\left( x\right) =\sum_{k=1}^{n-1}\binom{n}{k}\frac{n\left( n-1\right)
+k\left( k-1\right) }{n\left( n-1\right) }f\left( x_{n,k}\right) \frac{%
x^{k}\left( 1-x\right) ^{n-k}}{1+x^{2}}.
\end{equation*}
\end{example}

We now change our point of view: instead of starting with the polynomial $%
Q_{n-1}$ we just start with an increasing sequence 
\begin{equation*}
0=x_{n,0}<x_{n,1}<...<x_{n,n-1}<x_{n,n}=1.
\end{equation*}
We use equation (\ref{eqxnk}) to define $\frac{\gamma _{n-1,k-1}}{\gamma
_{n-1,k}}.$ Clearly (\ref{eqxnk}) is equivalent to 
\begin{equation*}
x_{n,k}\left( 1+\frac{\gamma _{n-1,k-1}}{\gamma _{n-1,k}}\right) =\frac{%
\gamma _{n-1,k-1}}{\gamma _{n-1,k}}\text{ and }\frac{\gamma _{n-1,k-1}}{%
\gamma _{n-1,k}}=\frac{x_{n,k}}{1-x_{n,k}}
\end{equation*}
which is a recursion formula for $\gamma _{n-1,k}$ provided we have defined $%
\gamma _{n-1,0}.$ Hence define 
\begin{equation}
\gamma _{n-1,k}:=\gamma _{n-1,0}\prod\limits_{l=1}^{k}\frac{1-x_{n,l}}{%
x_{n,l}}.  \label{defgamm}
\end{equation}
These remarks lead to the following statement:

\begin{proposition}
Let $0=x_{n,0}<x_{n,1}<...<x_{n,n-1}<x_{n,n}=1$. Let $\gamma _{n-1,0}>0$ and
define $\gamma _{n-1,k}$ by (\ref{defgamm}) for $k=1,...,n$ and define 
\begin{equation*}
Q_{n-1}\left( x\right) =\sum_{k=0}^{n-1}\gamma _{n-1,k}x^{k}\left(
1-x\right) ^{n-1-k}.
\end{equation*}
Then $Q_{n-1}$ satisfies property (W) and the operator 
\begin{equation*}
R_{n}f\left( x\right) =\sum_{k=0}^{n}f\left( x_{n,k}\right) \left( \gamma
_{n-1,k}+\gamma _{n-1,k-1}\right) \frac{x^{k}\left( 1-x\right) ^{n-k}}{%
Q_{n-1}\left( x\right) }
\end{equation*}
is the rational Bernstein operator $R_{n}$ fixing $e_{0}$ and $e_{1}$.
\end{proposition}

\begin{proof}
There is not much to show: by definition of $\gamma _{n-1,k}$ we see that $%
\frac{\gamma _{n-1,k-1}}{\gamma _{n-1,k}}=\frac{x_{n,k}}{1-x_{n,k}}$. This
clearly implies that (\ref{eqxnk}) holds, so the nodes of the operator $%
R_{n} $ are just the given numbers $x_{n,k}.$ Since $x_{n,k}$'s are
increasing we see that $\frac{\gamma _{n-1,k-1}}{\gamma _{n-1,k}}$ are
increasing and therefore property (W) holds.
\end{proof}

\section{Convergence of rational Bernstein operators}

The following result is of central importance:

\begin{theorem}
\label{ThmPowers}Let $R_{n}$ be the rational Bernstein operator for the
polynomial $Q_{n-1}\left( x\right) $ of degree $\leq n-1$ satisfying (i) and
(ii) in the introduction and let $x_{n,k}$ be defined by (\ref{eqxnk}). Then
the following identity 
\begin{equation}
R_{n}\left( e_{s}\right) \left( x\right) -x^{s}=\frac{x\left( 1-x\right) }{%
Q_{n-1}\left( x\right) }\sum_{l=0}^{s-2}x^{l}\sum_{k=0}^{n-1}\gamma
_{n-1,k}\left( x_{n,k+1}^{s-1-l}-x_{n,k}^{s-1-l}\right) x^{k}\left(
1-x\right) ^{n-1-k}  \label{eqneu}
\end{equation}
holds for the polynomial $e_{s}\left( x\right) =x^{s}$ and $s\geq 1.$
\end{theorem}

\begin{proof}
At first we note that (\ref{eqxnk}) implies that  
\begin{equation*}
\gamma _{n-1,k}+\gamma _{n-1,k-1}=\gamma _{n-1,k-1}\frac{1-x_{n,k}}{x_{n,k}}%
+\gamma _{n-1,k-1}=\gamma _{n-1,k-1}\frac{1}{x_{n,k}}.
\end{equation*}
It follows that 
\begin{equation*}
R_{n}f=f\left( 0\right) \gamma _{n-1,0}\frac{\left( 1-x\right) ^{n}}{%
Q_{n-1}\left( x\right) }+\sum_{k=1}^{n}f\left( x_{n,k}\right) \frac{\gamma
_{n-1,k-1}}{x_{n,k}}\frac{x^{k}\left( 1-x\right) ^{n-k}}{Q_{n-1}\left(
x\right) }.
\end{equation*}
Let $s\geq 1$ and $e_{s}\left( x\right) =x^{s}.$ Since $x_{n,0}=0$ and $%
e_{s}\left( x_{n,0}\right) =x_{n,0}^{s}=0$ we have 
\begin{equation}
Q_{n-1}\left( x\right) R_{n}\left( e_{s}\right) \left( x\right)
=\sum_{k=1}^{n}\gamma _{n-1,k-1}x_{n,k}^{s-1}\cdot x^{k}\left( 1-x\right)
^{n-k}.  \label{eqidqn1}
\end{equation}
Using an index transformation we arrive at 
\begin{equation}
Q_{n-1}\left( x\right) R_{n}\left( e_{s}\right) \left( x\right)
=x\sum_{k=0}^{n-1}\gamma _{n-1,k}x_{n,k+1}^{s-1}\cdot x^{k}\left( 1-x\right)
^{n-1-k}.  \label{eqNice}
\end{equation}
Writing $x^{k}\left( 1-x\right) ^{n-1-k}=x^{k}\left( 1-x\right)
^{n-k}+x^{k+1}\left( 1-x\right) ^{n-1-k}$ we obtain 
\begin{eqnarray*}
Q_{n-1}\left( x\right) R_{n}\left( e_{s}\right) \left( x\right) 
&=&x\sum_{k=0}^{n-1}\gamma _{n-1,k}x_{n,k+1}^{s-1}\cdot x^{k}\left(
1-x\right) ^{n-k} \\
&&+x\sum_{k=0}^{n-1}\gamma _{n-1,k}x_{n,k+1}^{s-1}\cdot x^{k+1}\left(
1-x\right) ^{n-1-k}.
\end{eqnarray*}
The second sum is equal to $x\sum_{k=1}^{n}\gamma
_{n-1,k-1}x_{n,k}^{s-1}\cdot x^{k}\left( 1-x\right) ^{n-k}.$ Using the
convention $\gamma _{n-1,n}=\gamma _{n-1,-1}=0$ and the fact that $x_{n,0}=0$
we obtain 
\begin{equation*}
Q_{n-1}\left( x\right) R_{n}\left( e_{s}\right) \left( x\right)
=x\sum_{k=0}^{n}\left( \gamma _{n-1,k}x_{n,k+1}^{s-1}+\gamma
_{n-1,k-1}x_{n,k}^{s-1}\right) \cdot x^{k}\left( 1-x\right) ^{n-k}.
\end{equation*}

On the other hand, we have the trivial identity 
\begin{equation*}
\gamma _{n-1,k}\left( x_{n,k+1}^{s-1}-x_{n,k}^{s-1}\right) =\gamma
_{n-1,k}\left( x_{n,k+1}^{s-1}+\frac{x_{n,k}}{1-x_{n,k}}x_{n,k}^{s-2}\left(
x_{n,k}-1\right) \right) 
\end{equation*}
and using $\gamma _{n-1,k}=\gamma _{n-1,k-1}\frac{1-x_{n,k}}{x_{n,k}}$    
\begin{eqnarray*}
\gamma _{n-1,k}\left( x_{n,k+1}^{s-1}-x_{n,k}^{s-1}\right)  &=&\gamma
_{n-1,k}\left( x_{n,k+1}^{s-1}+\frac{\gamma _{n-1,k-1}}{\gamma _{n-1,k}}%
x_{n,k}^{s-2}\left( x_{n,k}-1\right) \right)  \\
&=&\gamma _{n-1,k}x_{n,k+1}^{s-1}+\gamma _{n-1,k-1}x_{n,k}^{s-1}-\gamma
_{n-1,k-1}x_{n,k}^{s-2}.
\end{eqnarray*}
It follows that 
\begin{eqnarray*}
Q_{n-1}\left( x\right) R_{n}\left( e_{s}\right) \left( x\right) 
&=&x\sum_{k=0}^{n}\gamma _{n-1,k}\left( x_{n,k+1}^{s-1}-x_{n,k}^{s-1}\right)
\cdot x^{k}\left( 1-x\right) ^{n-k} \\
&&+x\sum_{k=0}^{n}\gamma _{n-1,k-1}x_{n,k}^{s-2}\cdot x^{k}\left( 1-x\right)
^{n-k}.
\end{eqnarray*}
As $\gamma _{n-1,n}=0,$ the indices in the first sum range only up to $n-1$.
The first summand of the second sum is zero. Using (\ref{eqidqn1}) for $s-1$
instead of $s$ for the second sum we arrive 
\begin{eqnarray*}
Q_{n-1}\left( x\right) R_{n}\left( e_{s}\right) \left( x\right)  &=&x\left(
1-x\right) \sum_{k=0}^{n-1}\gamma _{n-1,k}\left(
x_{n,k+1}^{s-1}-x_{n,k}^{s-1}\right) \cdot x^{k}\left( 1-x\right) ^{n-1-k} \\
&&+x\cdot Q_{n-1}\left( x\right) R_{n}\left( e_{s-1}\right) \left( x\right) .
\end{eqnarray*}
Now use this formula inductively and recall that $R_{n}\left( e_{1}\right)
=e_{1}$ leading to the statement in theorem.
\end{proof}

\begin{corollary}
\label{CorEstx2}The rational Bernstein operators $R_{n}$ satisfy the
inequality 
\begin{equation*}
\left\vert R_{n}\left( e_{2}\right) \left( x\right) -x^{2}\right\vert \leq
\sup_{0\leq k\leq n-1}\left\vert x_{n,k+1}-x_{n,k}\right\vert \cdot x\left(
1-x\right) .
\end{equation*}
\end{corollary}

\begin{proof}
From Theorem \ref{ThmPowers} for $s=2$ we see that 
\begin{equation}
R_{n}\left( e_{2}\right) \left( x\right) -x^{2}=\frac{x\left( 1-x\right) }{%
Q_{n-1}\left( x\right) }\sum_{k=0}^{n-1}\gamma _{n-1,k}\left(
x_{n,k+1}-x_{n,k}\right) x^{k}\left( 1-x\right) ^{n-1-k}  \label{eqRQ}
\end{equation}
and then the statement is obvious since $\gamma _{n-1,k}$ is positive. 
\end{proof}

\begin{corollary}
The rational Bernstein operators $R_{n}$ converges to the identity operator
if and only if 
\begin{equation}
\Delta _{n}:=\sup_{0\leq k\leq n-1}\left\vert x_{n,k+1}-x_{n,k}\right\vert
\label{eqDeltan2}
\end{equation}
converges to $0.$
\end{corollary}

\begin{proof}
If $\Delta _{n}$ converges to zero it follows that $R_{n}e_{2}$ converges
uniformly to $e_{2}$ and Korovkin's theorem shows that $R_{n}$ converges to
the identity operator. Conversely, suppose that $R_{n}$ converges to the
identity operator and suppose that $\Delta _{n}$ does not converge to $0.$
Then there exists $\delta >0$ and a subsequence $\left( n_{l}\right) _{l}$
such that $\Delta _{n_{l}}\geq 2\delta .$ Hence for each $l$ there $%
k_{n,l}\in \left\{ 0,...,n_{l}-1\right\} $ such that 
\begin{equation}
\left\vert x_{n_{l},k_{l}+1}-x_{n_{l},k_{l}}\right\vert \geq \delta .
\label{eqdel}
\end{equation}
Since $x_{n,k}\in \left[ 0,1\right] $ we can pass to a subsequence of $%
x_{n_{l},k_{l}}$ which converges to some point $x_{0}$ and we can pass again
to a subsequence of the subsequence such that $x_{n_{l_{r}},k_{l_{r}}}$
converges to $x_{0}$ and $x_{n_{l_{r}},k_{l_{r}}+1}$ converges to $x_{1}.$
From (\ref{eqdel}) it follows that $\left\vert x_{1}-x_{0}\right\vert \geq
\delta ,$ and since $x_{n_{l},k_{l}}\leq x_{n_{l},k_{l}+1}$ we infer that $%
x_{0}\leq x_{1}.$ Now we take a natural number $r_{0}$ such that $\left\vert
x_{0}-x_{n_{l_{r}},k_{l_{r}}}\right\vert <\delta /3$ and $\left\vert
x_{1}-x_{n_{l_{r}},k_{l_{r}}+1}\right\vert <\delta /3$ for all $r\geq r_{0}.$
From the monotonicity of $x_{n,k}$ for $k=0,...,n_{l}-1$ it follows that $%
x_{n_{l_{r}},k}\notin \left[ x_{0}+\delta /3,x_{1}-\delta /3\right] $ for
all $k=0,...,n_{l_{r}}$ and $l\geq l_{0}.$ Now construct a continuous
non-zero function $f$ with support in $\left[ x_{0}+\delta /3,x_{1}-\delta /3%
\right] $ such that $f\left( \xi \right) \neq 0$ for some $\xi \in \left[
x_{0}+\delta /3,x_{1}-\delta /3\right] .$ Then $B_{n_{l}}f\left( x\right) =0$
for all $x\in \left[ 0,1\right] .$ By assumption, $B_{n\,l_{r}}f\left( \xi
\right) $ converges to $f\left( \xi \right) \neq 0.$ Since $%
B_{n\,l_{r}}f\left( \xi \right) =0$ we obtain a contradiction completing the
proof.
\end{proof}

\begin{corollary}
The following inequality holds for all $x\in \left[ 0,1\right] $ and for all
natural numbers $s\geq 2:$ 
\begin{equation*}
0\leq x^{s}<R_{n}\left( e_{s}\right) \left( x\right) .
\end{equation*}
\end{corollary}

\begin{proof}
The right hand side in (\ref{eqneu}) is strictly positive for $x\in \left[
0,1\right] $ and $s\geq 2.$ Alternatively, one may argue that the function $%
e_{s}$ is convex, and the result follows from the remarks in \cite[p. 46]
{PiSa09}.
\end{proof}

In the rest of this section we shall prove some inequalities which will be
needed in Section 5:

\begin{proposition}
\label{Prop3}The following inequality holds 
\begin{equation*}
0\leq R_{n}\left( e_{3}\right) \left( x\right) -x^{3}\leq 3\cdot \left(
R_{n}\left( e_{2}\right) \left( x\right) -x^{2}\right) .
\end{equation*}
\end{proposition}

\begin{proof}
From Theorem \ref{ThmPowers} for $s=3$ we see that 
\begin{equation*}
R_{n}\left( e_{3}\right) \left( x\right) -x^{3}=\frac{x\left( 1-x\right) }{%
Q_{n-1}\left( x\right) }\sum_{k=0}^{n-1}\gamma _{n-1,k}A_{k}x^{k}\left(
1-x\right) ^{n-1-k}
\end{equation*}
where 
\begin{equation*}
A_{k}=x_{n,k+1}^{2}-x_{n,k}^{2}+x\left( x_{n,k+1}-x_{n,k}\right) =\left(
x_{n,k+1}-x_{n,k}\right) \left( x_{n,k+1}+x_{n,k}+x\right) \geq 0.
\end{equation*}
Since $0\leq x_{n,k+1}+x_{n,k}+x\leq 3$ we obtain 
\begin{equation*}
0\leq R_{n}\left( e_{3}\right) \left( x\right) -x^{3}\leq 3\frac{x\left(
1-x\right) }{Q_{n-1}\left( x\right) }\sum_{k=0}^{n-1}\gamma _{n-1,k}\left(
x_{n,k+1}-x_{n,k}\right) x^{k}\left( 1-x\right) ^{n-1-k}
\end{equation*}
and the last expression is equal to $3\left( R_{n}\left( e_{2}\right) \left(
x\right) -x^{2}\right) .$ The proof is complete.
\end{proof}

\begin{proposition}
\label{PropRM}Let $r$ be a natural number. Then the expression 
\begin{equation*}
A:=\frac{x}{Q_{n-1}\left( x\right) }\sum_{k=0}^{n-1}\left(
x-x_{n,k+1}\right) ^{r}\gamma _{n-1,k}x^{k}\left( 1-x\right) ^{n-1-k}
\end{equation*}
is equal to 
\begin{equation*}
B:=\sum_{l=0}^{r}\binom{r}{l}x^{r-l}\left( -1\right) ^{l}\left[ R_{n}\left(
e_{l+1}\right) \left( x\right) -x^{l+1}\right] .
\end{equation*}
\end{proposition}

\begin{proof}
Since $\left( x-x_{n,k+1}\right) ^{r}=\sum_{l=0}^{r}\binom{r}{l}%
x^{r-l}\left( -1\right) ^{l}x_{n,k+1}^{l}$ it is easy to see that 
\begin{equation*}
A=\sum_{l=0}^{r}\binom{r}{l}x^{r-l}\left( -1\right) ^{l}\frac{x}{%
Q_{n-1}\left( x\right) }\sum_{k=0}^{n-1}\gamma
_{n-1,k}x_{n,k+1}^{l}x^{k}\left( 1-x\right) ^{n-1-k}.
\end{equation*}
Using (\ref{eqNice}) we see that 
\begin{equation*}
A=\sum_{l=0}^{r}\binom{r}{l}x^{r-l}\left( -1\right) ^{l}R_{n}\left(
e_{l+1}\right)
\end{equation*}
and the result follows from the fact that 
\begin{equation*}
\sum_{l=0}^{r}\binom{r}{l}x^{r-l}\left( -1\right) ^{l}x^{l+1}=x\left(
x+\left( -x\right) \right) ^{r}=0.
\end{equation*}
\end{proof}

For the Bernstein operator $B_{n}$ it is well known that  
\begin{equation*}
B_{n}e_{2}\left( x\right) -x^{2}=\frac{x\left( 1-x\right) }{n}
\end{equation*}
is a polynomial of degree $\leq 2.$ For rational Bernstein operators the
expression $R_{n}e_{2}\left( x\right) -x^{2}$ is a never polynomial except
that $Q_{n-1}\left( x\right) =1,$ the case of the classical Bernstein
operator. Indeed, suppose that $R_{n}e_{2}\left( x\right) -x^{2}=p_{s}\left(
x\right) $ for some polynomial $p_{s}\left( x\right) $ of degree $s.$ Then
by (\ref{eqRQ}) 
\begin{equation*}
x\left( 1-x\right) \sum_{k=0}^{n-1}\gamma _{n-1,k}\left(
x_{n,k+1}-x_{n,k}\right) x^{k}\left( 1-x\right) ^{n-1-k}=p_{s}\left(
x\right) Q_{n-1}\left( x\right) 
\end{equation*}
which shows that $p_{s}\left( x\right) Q_{n-1}\left( x\right) $ has degree $%
\leq n+1.$ Hence $s\leq 2$ and clearly $x\left( 1-x\right) $ must be a
factor of $p_{s}\left( x\right) .$ Hence $p_{s}\left( x\right) =Ax\left(
1-x\right) .$ By uniqueness of the representation (\ref{eqQrep}) we infer
that $\gamma _{n-1,k}\left( x_{n,k+1}-x_{n,k}\right) =A\gamma _{n-1,k},$ so $%
x_{n,k+1}-x_{n,k}=A,$ and we arrive at the classical Bernstein operator.

\section{Error estimates for rational Bernstein operators}

Next we derive quantitative convergence results for $R_{n}.$ By estimates of
O. Shisha and B. Mond (we refer to Theorem 8.1 in \cite{PiSa09}) we conclude
that 
\begin{equation}
\left\vert R_{n}f\left( x\right) -f\left( x\right) \right\vert \leq \left( 1+%
\frac{1}{h}\sqrt{R_{n}\left( e_{1}-x\right) ^{2}\left( x\right) }\right)
\omega _{1}\left( f,h\right)  \label{eqSiMo}
\end{equation}
for all $f\in C\left[ 0,1\right] $ and $h>0$ where $\omega _{1}\left(
f,h\right) $ is the first modulus of continuity defined by 
\begin{equation*}
\omega _{1}\left( f,h\right) =\sup_{\left\vert x-y\right\vert \leq
h}\left\vert f\left( x\right) -f\left( y\right) \right\vert .
\end{equation*}
Since 
\begin{equation*}
R_{n}\left( e_{1}-x\right) ^{2}\left( x\right) =R_{n}\left( e_{2}\right)
\left( x\right) -2xR_{n}e_{1}\left( x\right) +x^{2}=R_{n}e_{2}\left(
x\right) -x^{2}
\end{equation*}
we obtain from (\ref{eqSiMo}) for $h:=\sqrt{\Delta _{n}}$, defined in (\ref
{eqDeltan2}), and from Corollary \ref{CorEstx2} the following result:

\begin{theorem}
The rational Bernstein operators $R_{n}$ satisfies the following inequality: 
\begin{equation}
\left\vert R_{n}f\left( x\right) -f\left( x\right) \right\vert \leq \left( 1+%
\sqrt{x\left( 1-x\right) }\right) \omega _{1}\left( f,\sqrt{\Delta _{n}}%
\right)  \label{eqQuant1}
\end{equation}
for all $f\in C\left[ 0,1\right] .$
\end{theorem}

Similarly, Theorem 8.2 in \cite{PiSa09} provides us with the estimate 
\begin{equation*}
\left\vert R_{n}f\left( x\right) -f\left( x\right) \right\vert \leq \left( 1+%
\frac{1}{2h^{2}}R_{n}\left( e_{1}-x\right) ^{2}\left( x\right) \right)
\omega _{2}\left( f,h\right)
\end{equation*}
for all $f\in C\left[ 0,1\right] $ and $h>0$ where $\omega _{2}\left(
f,h\right) $ is the second modulus of continuity defined by 
\begin{equation*}
\omega _{2}\left( f,h\right) =\sup_{\left\vert \delta \right\vert \leq
h}\left\{ \left\vert f\left( x+\delta \right) -2f\left( x\right) +f\left(
x-\delta \right) \right\vert :x\pm h\in \left[ a,b\right] \right\} .
\end{equation*}
Taking $h=\sqrt{\Delta _{n}}$ we obtain

\begin{theorem}
The rational Bernstein operators $R_{n}$ satisfy the following inequality 
\begin{equation}
\left\vert R_{n}f\left( x\right) -f\left( x\right) \right\vert \leq \left( 1+%
\frac{1}{2}x\left( 1-x\right) \right) \omega _{2}\left( f,\sqrt{\Delta _{n}}%
\right)  \label{eqQuant2}
\end{equation}
for all $f\in C\left[ 0,1\right] .$
\end{theorem}

\section{Voronovskaja's Theorem}

The classical Voronovskaja theorem states the following:

\begin{theorem}
Let $f:\left[ 0,1\right] \rightarrow \mathbb{R}$ be bounded and
differentiable in a neighborhood of $x$ and has second derivative $f^{\prime
\prime }\left( x\right) .$ Then 
\begin{equation*}
\lim_{n\rightarrow \infty }n\cdot \left( R_{n}f\left( x\right) -f\left(
x\right) \right) =\frac{x\left( 1-x\right) }{2}f^{\prime \prime }\left(
x\right) .
\end{equation*}
\end{theorem}

We shall need the following generalization due to R.G. Mamedov \cite{Mame62}%
, see also \cite{GoTa09} and \cite{Tach12} for quantitative estimates and
higher order of differentiability.

\begin{theorem}
\label{ThmMame}Let $f\in C^{2}\left[ 0,1\right] $ and $L_{n}:C\left[ 0,1%
\right] \rightarrow C\left[ 0,1\right] $ be a sequence of positive operators
such that $L_{n}e_{j}=e_{j}$ for $j=0,1$ and 
\begin{equation*}
\lim_{n\rightarrow \infty }\frac{L_{n}\left( e_{1}-x\right) ^{4}\left(
x\right) }{L_{n}\left( e_{1}-x\right) ^{2}\left( x\right) }=0
\end{equation*}
for each $x\in \left[ 0,1\right] .$ Then 
\begin{equation*}
\frac{L_{n}f\left( x\right) -f\left( x\right) }{L_{n}\left( e_{1}-x\right)
^{2}\left( x\right) }\rightarrow \frac{1}{2}f^{\prime \prime }\left( x\right)
\end{equation*}
when $n\rightarrow \infty $.
\end{theorem}

The classical proof of the Voronovskaja theorem requires the computation of
the moments of order $r$ of the Bernstein operator $B_{n}$: 
\begin{equation*}
B_{n}\left[ \left( e_{1}-x\right) ^{r}\right] \left( x\right)
=\sum_{k=0}\left( \frac{k}{n}-x\right) ^{r}\binom{n}{k}x^{k}\left(
1-x\right) ^{n-k}=:\frac{1}{n^{r}}T_{n,r}\left( x\right) .
\end{equation*}
It is well known that $T_{n,r}\left( x\right) $ is a polynomial of degree $r$
in the variable $x$ and one can determine $T_{n,r}\left( x\right) $
recursively by the formula 
\begin{equation*}
T_{n,r+1}\left( x\right) =x\left( 1-x\right) \left[ T_{n,r}^{\prime }\left(
x\right) +nsT_{n,r-1}\left( x\right) \right] ,
\end{equation*}
see \cite{Lore86}. From this it is not difficult to show that for each $r\in 
\mathbb{N}$ there exists a constant $A_{r}>0$ such that 
\begin{equation}
B_{n}\left[ \left( e_{1}-x\right) ^{r}\right] \left( x\right) \leq \sqrt{%
A_{r}}\frac{1}{\sqrt{n}^{r}},  \label{eqTach}
\end{equation}
see e.g. \cite{Tach12}. In passing we mention that in the recent article 
\cite{GaIv12} the following inequality was established: for $r\in \mathbb{N}$
there exists a constant $K_{r}>0$ such that 
\begin{equation*}
B_{n}\left[ \left( e_{1}-x\right) ^{r+1}\right] \left( x\right) \leq \frac{%
K_{r}}{\sqrt{n}}B_{n}\left[ \left( e_{1}-x\right) ^{r}\right] \left( x\right)
\end{equation*}
which clearly implies (\ref{eqTach}).

In the case of the rational Bernstein operator the moments $R_{n}\left[
\left( e_{1}-x\right) ^{r}\right] \left( x\right) $ are not polynomials in
the variable $x$ as we have seen already at the end of Section 4 for $r=2.$
Nonetheless, we can compute them explicitly but the formulae are much more
complicated. Indeed, if we use the binomial theorem for $\left(
e_{1}-x\right) ^{r}$ we obtain 
\begin{equation*}
R_{n}\left[ \left( e_{1}-x\right) ^{r}\right] \left( x\right) =\sum_{s=0}^{r}%
\binom{r}{s}\left( -x\right) ^{r-s}R_{n}\left( e_{s}\right) \left( x\right)
\end{equation*}
and since $0=\left( x-x\right) ^{r}=\sum_{s=0}^{r}\binom{r}{s}\left(
-x\right) ^{r-s}x^{s}$ we have 
\begin{equation}
R_{n}\left[ \left( e_{1}-x\right) ^{r}\right] \left( x\right) =\sum_{s=2}^{r}%
\binom{r}{s}\left( -x\right) ^{r-s}\left[ R_{n}\left( e_{s}\right) \left(
x\right) -x^{s}\right]  \label{eqmom}
\end{equation}
where we used the fact that $R_{n}\left( e_{s}\right) =e_{s}$ for $s=0,1.$
Theorem \ref{ThmPowers} provides then an explicit formula for the moments.
But in view of Theorem \ref{ThmMame} we have to estimate 
\begin{equation*}
\frac{R_{n}\left( e_{1}-x\right) ^{4}\left( x\right) }{R_{n}\left(
e_{1}-x\right) ^{2}\left( x\right) }
\end{equation*}
and it is therefore not sufficient just to estimate the moments.

\begin{theorem}
\label{ThmMom4}The fourth moment satisfies the following inequality; 
\begin{equation*}
R_{n}\left( e_{1}-x\right) ^{4}\left( x\right) \leq \Delta _{n}\cdot \left[
R_{n}\left( e_{1}-x\right) ^{2}\left( x\right) \right] \left(
6x^{2}-15x+12+\Delta _{n}\right) .
\end{equation*}
\end{theorem}

\begin{proof}
Formula (\ref{eqmom}) shows that $R_{n}\left( e_{1}-x\right) ^{4}\left(
x\right) $ is equal to 
\begin{equation*}
\left( R_{n}\left( e_{4}\right) \left( x\right) -x^{4}\right) -4x\left(
R_{n}\left( e_{3}\right) \left( x\right) -x^{3}\right) +6x^{2}\left(
R_{n}\left( e_{2}\right) \left( x\right) -x^{2}\right) .
\end{equation*}
By Theorem \ref{ThmPowers} we can calculate each summand explicitly and we
obtain 
\begin{equation}
R_{n}\left( e_{1}-x\right) ^{4}\left( x\right) =\frac{x\left( 1-x\right) }{%
Q_{n-1}\left( x\right) }\sum_{k=0}^{n-1}\gamma _{n-1,k}x^{k}\left(
1-x\right) ^{n-1-k}\cdot H_{k}  \label{eqmom4}
\end{equation}
with 
\begin{eqnarray*}
H_{k} &=&x_{n,k+1}^{3}-x_{n,k}^{3}+x\left( x_{n,k+1}^{2}-x_{n,k}^{2}\right)
+x^{2}\left( x_{n,k+1}-x_{n,k}\right)  \\
&&-4x\left( x_{n,k+1}^{2}-x_{n,k}^{2}+x\left( x_{n,k+1}-x_{n,k}\right)
\right) +6x^{2}\left( x_{n,k+1}-x_{n,k}\right) 
\end{eqnarray*}
which simplifies to
\begin{equation*}
H_{k}=\left( x_{n,k+1}^{3}-x_{n,k}^{3}\right) -3x\left(
x_{n,k+1}^{2}-x_{n,k}^{2}\right) +3x^{2}\left( x_{n,k+1}-x_{n,k}\right) .
\end{equation*}
We write $H_{k}=\left( x_{n,k+1}-x_{n,k}\right) A_{k}$ with 
\begin{equation*}
A_{k}=x_{n,k+1}^{2}+x_{n,k+1}x_{n,k}+x_{n,k}^{2}-3x\left(
x_{n,k+1}+x_{n,k}\right) +3x^{2}.
\end{equation*}
A straightforward calculation shows that
\begin{equation*}
A_{k}=3\left( x-\frac{1}{2}\left( x_{n,k+1}+x_{n,k}\right) \right) ^{2}+%
\frac{1}{4}\left( x_{n,k+1}-x_{n,k}\right) ^{2}\geq 0.
\end{equation*}
Hence $A_{k}$ is positive and and it is easy to see that 
\begin{equation}
R_{n}\left( e_{1}-x\right) ^{4}\left( x\right) \leq \Delta _{n}\frac{x\left(
1-x\right) }{Q_{n-1}\left( x\right) }\sum_{k=0}^{n-1}\gamma
_{n-1,k}x^{k}\left( 1-x\right) ^{n-1-k}\cdot A_{k}.  \label{eqR1}
\end{equation}
We write now 
\begin{equation}
A_{k}=3\left( x-x_{n,k+1}\right) ^{2}+3\left( x-x_{n,k+1}\right) \left(
x_{n,k+1}-x_{n,k}\right) +\left( x_{n,k+1}-x_{n,k}\right) ^{2}.
\label{eqR1b}
\end{equation}
Proposition \ref{PropRM} applied to the case $r=2$ and \ref{Prop3} show that 
\begin{eqnarray*}
&&\frac{x}{Q_{n-1}\left( x\right) }\sum_{k=0}^{n-1}\gamma _{n-1,k}\left(
x-x_{n,k+1}\right) ^{2}x^{k}\left( 1-x\right) ^{n-1-k} \\
&=&R_{n}\left( e_{3}\right) \left( x\right) -x^{3}-2x\left[ R_{n}\left(
e_{2}\right) \left( x\right) -x^{2}\right] \leq \left( 3-2x\right) \left[
R_{n}\left( e_{2}\right) \left( x\right) -x^{2}\right] .
\end{eqnarray*}
Formula (\ref{eqR1}) and (\ref{eqR1b}) in connection with the last
inequality and the simple estimates $\left| x-x_{n,k+1}\right| \leq 1$ and $%
x_{n,k+1}-x_{n,k}\leq \Delta _{n}$ lead to 
\begin{eqnarray*}
R_{n}\left( e_{1}-x\right) ^{4}\left( x\right)  &\leq &\Delta _{n}\left(
1-x\right) \cdot 3\left( 3-2x\right) \cdot \left[ R_{n}\left( e_{1}\right)
\left( x\right) -x^{2}\right]  \\
&&+\Delta _{n}\left( 3+\Delta _{n}\right) \frac{x\left( 1-x\right) }{%
Q_{n-1}\left( x\right) }\sum_{k=0}^{n-1}\gamma _{n-1,k}\left(
x_{n,k+1}-x_{n,k}\right) x^{k}\left( 1-x\right) ^{n-1-k}.
\end{eqnarray*}
It follows that 
\begin{equation*}
R_{n}\left( e_{1}-x\right) ^{4}\left( x\right) \leq \Delta _{n}\left[
R_{n}\left( e_{2}\right) \left( x\right) -x^{2}\right] \left( \left(
1-x\right) \left( 9-6x\right) +3+\Delta _{n}\right) 
\end{equation*}
and the statement is now obvious since $R_{n}\left( e_{2}\right) \left(
x\right) -x^{2}=R\left( e_{1}-x\right) ^{2}\left( x\right) .$
\end{proof}

Using Theorem \ref{ThmMame} and Theorem \ref{ThmMom4} we obtain

\begin{theorem}
Let $f\in C^{2}\left[ 0,1\right] $ and assume that $\Delta _{n}\rightarrow 0$
for the rational Bernstein operators $R_{n}:C\left[ 0,1\right] \rightarrow C%
\left[ 0,1\right] .$ Then 
\begin{equation*}
\frac{R_{n}f\left( x\right) -f\left( x\right) }{R_{n}\left( e_{1}-x\right)
^{2}\left( x\right) }\rightarrow \frac{1}{2}f^{\prime \prime }\left(
x\right) .
\end{equation*}
\end{theorem}

\section{Special classes of rational Bernstein operators}

In \cite{PiSa09} error estimates and convergence results have been given for
rational Bernstein operators $R_{n}$ under the assumption that there exists
a positive function $\varphi \in C\left[ 0,1\right] $ such that 
\begin{equation*}
Q_{n-1}\left( x\right) :=B_{n-1}\varphi \left( x\right)
=\sum_{k=0}^{n-1}\varphi \left( \frac{k}{n-1}\right) \binom{n-1}{k}%
x^{k}\left( 1-x\right) ^{n-1-k}
\end{equation*}
where $B_{n-1}$ is the classical Bernstein operator of degree $n-1$. Then $%
Q_{n-1}$ has clearly positive Bernstein coefficients but in general one has
to assume in addition that property (W) is satisfied.

It is shown in \cite[p. 42]{PiSa09} that property (W) is satisfied provided
that $n$ is sufficiently large and $\varphi \in C^{2}\left[ 0,1\right] .$
Later in this section we shall show that it suffices to assume only that $%
\varphi \in C^{1}\left[ 0,1\right] ,$ and we shall show by example that the
result is not true for a Lipschitz function. Now we cite from \cite{PiSa09}
the following result:

\begin{theorem}
\label{ThmPisa}Suppose that $\varphi \in C\left[ 0,1\right] $ such that $%
Q_{n-1}\left( x\right) =B_{n-1}\varphi \left( x\right) $ satisfies property
(W). Then 
\begin{equation*}
\left\vert R_{n}f\left( x\right) -f\left( x\right) \right\vert \leq \left( 1+%
\frac{1}{2}\sqrt{\frac{\max_{x\in \left[ 0,1\right] }\varphi \left( x\right) 
}{\min_{x\in \left[ 0,1\right] }\varphi \left( x\right) }}\right) \omega
_{1}\left( f,\frac{1}{\sqrt{n}}+\frac{1}{2m}\omega _{1}\left( \varphi ,\frac{%
1}{n-1}\right) \right) .
\end{equation*}
\end{theorem}

We want to show that Theorem \ref{ThmPisa} can be derived and improved from
our previous results. Indeed we want to show:

\begin{theorem}
\label{ThmErrorPhi}Suppose that $\varphi \in C\left[ 0,1\right] $ such that $%
Q_{n-1}\left( x\right) =B_{n-1}\varphi \left( x\right) $ satisfies property
(W). Then 
\begin{equation*}
\left\vert R_{n}f\left( x\right) -f\left( x\right) \right\vert \leq \left( 1+%
\sqrt{x\left( 1-x\right) }\right) \omega _{1}\left( f,\frac{1}{\sqrt{n}}+%
\frac{1}{2m}\omega _{1}\left( \varphi ,\frac{1}{n-1}\right) \right) .
\end{equation*}
\end{theorem}

Obviously the result is better since $\sqrt{x\left( 1-x\right) }\leq 1/2$
and $\min_{x\in \left[ 0,1\right] }\varphi \left( x\right) \leq \max_{x\in %
\left[ 0,1\right] }\varphi \left( x\right) .$ We need the following result
which is implicitly contained in \cite{PiSa09}:

\begin{proposition}
\label{PropPsi}Let $\varphi \in C\left[ 0,1\right] $ positive and $%
Q_{n-1}\left( x\right) =B_{n-1}\varphi \left( x\right)
=\sum_{k=0}^{n-1}\gamma _{n-1,k}x^{k}$ with $\gamma _{n-1,k}=\varphi \left(
k/(n-1)\right) \binom{n-1}{k}.$ If one defines 
\begin{equation*}
x_{n,k}:=\frac{\gamma _{n-1,k-1}}{\gamma _{n-1,k-1}+\gamma _{n-1,k}}=\frac{%
k\varphi \left( \frac{k-1}{n-1}\right) }{k\varphi \left( \frac{k-1}{n-1}%
\right) +\left( n-k\right) \varphi \left( \frac{k}{n-1}\right) }
\end{equation*}
then 
\begin{equation}
\Delta _{n}=\sup_{k=0,....n-1}\left\vert x_{n,k+1}-x_{n,k}\right\vert \leq 
\frac{1}{2m}\omega _{1}\left( \varphi ,\frac{1}{n-1}\right) +\frac{1}{n},
\label{eqPS}
\end{equation}
where $m=\min_{x\in \left[ 0,1\right] }\varphi \left( x\right) .$
\end{proposition}

\begin{proof}
Define 
\begin{equation*}
\psi _{h}\left( x\right) =\frac{x\varphi \left( x-h\right) }{x\varphi \left(
x-h\right) +\left( 1-x+h\right) \varphi \left( x\right) }.
\end{equation*}
Put $h=1/\left( n-1\right) $ and $x=k/\left( n-1\right) $ then 
\begin{equation}
x_{n,k}=\psi _{\frac{1}{n-1}}\left( \frac{k}{n-1}\right) .  \label{eqx1}
\end{equation}
Similarly, 
\begin{equation}
\frac{k}{n}=\tau _{\frac{1}{n-1}}\left( \frac{k}{n-1}\right) \text{ for }%
\tau _{h}\left( x\right) =\frac{x}{1+h}.  \label{eqx2}
\end{equation}
We want to estimate $x_{n,k}-\frac{k}{n}$ and therefore we look at 
\begin{eqnarray*}
\psi _{h}\left( x\right) -\frac{x}{1+h} &=&x\frac{\left( 1+h\right) \varphi
\left( x-h\right) -x\varphi \left( x-h\right) -\left( 1-x+h\right) \varphi
\left( x\right) }{\left( 1+h\right) \left( x\varphi \left( x-h\right)
+\left( 1-x+h\right) \varphi \left( x\right) \right) } \\
&=&\frac{x\cdot \left( 1-x+h\right) \cdot \left( \varphi \left( x-h\right)
-\varphi \left( x\right) \right) }{\left( 1+h\right) \cdot \left( x\varphi
\left( x-h\right) +\left( 1-x+h\right) \varphi \left( x\right) \right) }.
\end{eqnarray*}
Further we can estimate with $m:=\min_{y\in \left[ 0,1\right] }\varphi
\left( y\right) $%
\begin{equation*}
x\varphi \left( x-h\right) +\left( 1-x+h\right) \varphi \left( x\right) \geq
\left( 1+h\right) m
\end{equation*}
and we obtain that 
\begin{equation*}
\left\vert \psi _{h}\left( x\right) -\frac{x}{1+h}\right\vert \leq \frac{%
x\left( 1-x+h\right) }{\left( 1+h\right) ^{2}m}\omega _{1}\left( \varphi
,h\right) \leq \frac{1}{4m}\omega _{1}\left( \varphi ,h\right) .
\end{equation*}
where we used that $4x\left( 1-x+h\right) \leq \left( 1+h\right) ^{2}$ for
all $x\in \left[ 0,1\right] $ and $h>0.$ Using (\ref{eqx1}) and (\ref{eqx2})
it follows that for all $k=0,....n$ and all $n$ the following inequality 
\begin{equation*}
\left\vert x_{n,k}-\frac{k}{n}\right\vert \leq \frac{1}{4m}\omega _{1}\left(
\varphi ,\frac{1}{n-1}\right)
\end{equation*}
holds. Since $x_{n,k+1}-x_{n,k}=x_{n,k+1}-\frac{k+1}{n}+\frac{1}{n}+\frac{k}{%
n}-x_{n,k}$ we can estimate 
\begin{equation}
\left\vert x_{n,k+1}-x_{n,k}\right\vert \leq \frac{1}{2m}\omega _{1}\left(
\varphi ,\frac{1}{n-1}\right) +\frac{1}{n}.  \label{eqlast1}
\end{equation}
\end{proof}

\textbf{Proof of Theorem \ref{ThmErrorPhi}:} Formula (\ref{eqPS}), the
inequality $\sqrt{a+b}\leq \sqrt{a}+\sqrt{b}$ for positive numbers $a,b,$
and (\ref{eqlast1}) imply that 
\begin{equation*}
\sqrt{\Delta _{n}}\leq \frac{1}{\sqrt{n}}+\sqrt{\frac{\omega _{1}\left(
\varphi ,\frac{1}{n-1}\right) }{2m}}=\frac{1}{\sqrt{n}}+\frac{1}{\sqrt{2m}}%
\sup_{\left\vert x-y\right\vert \leq \frac{1}{n-1}}\sqrt{\left\vert \varphi
\left( x\right) -\varphi \left( y\right) \right\vert }.
\end{equation*}
Further 
\begin{equation*}
\sup_{\left\vert x-y\right\vert \leq \frac{1}{n-1}}\sqrt{\left\vert \varphi
\left( x\right) -\varphi \left( y\right) \right\vert }=\sup_{\left\vert
x-y\right\vert \leq \frac{1}{n-1}}\frac{\left\vert \varphi \left( x\right)
-\varphi \left( y\right) \right\vert }{\sqrt{\left\vert \varphi \left(
x\right) +\varphi \left( y\right) \right\vert }}\leq \frac{\omega _{1}\left(
\varphi ,\frac{1}{n-1}\right) }{\sqrt{2m}}.
\end{equation*}
and $\sqrt{\Delta _{n}}\leq \frac{1}{\sqrt{n}}+\frac{1}{2m}\omega _{1}\left(
\varphi ,\frac{1}{n-1}\right) .$ Further (\ref{eqQuant1}) and the trivial
estimate $\omega _{1}\left( f,\delta \right) \leq \omega _{1}\left( f,\delta
^{\prime }\right) $ for $\delta \leq \delta ^{\prime }$ leads to 
\begin{equation*}
\left\vert R_{n}f\left( x\right) -f\left( x\right) \right\vert \leq \left( 1+%
\sqrt{x\left( 1-x\right) }\right) \omega _{1}\left( f,\frac{1}{\sqrt{n}}+%
\frac{1}{2m}\omega _{1}\left( \varphi ,\frac{1}{n-1}\right) \right)
\end{equation*}
which is the above estimate.

Finally we shall prove:

\begin{theorem}
Let $\varphi \in C\left[ 0,1\right] $ be strictly positive. If $\varphi \in
C^{1}\left[ 0,1\right] $ then $Q_{n-1}\left( x\right) :=B_{n-1}\varphi
\left( x\right) $ satisfies property (W) for sufficiently large $n\in 
\mathbb{N}.$ If $\varphi $ is Lipschitz continuous then $a+\varphi $
satisfies property (W) for sufficiently large $n\in \mathbb{N}$ and
sufficiently large $a>0.$
\end{theorem}

\begin{proof}
We use the notations from the proof of Proposition \ref{PropPsi}. In view of
(\ref{eqx1}) it suffices to show that the function $x\longmapsto \psi
_{h}\left( x\right) $ is increasing if $h>0$ is sufficiently small, or
equivalently, that for $\delta >0$ and $h>0$ sufficiently small and for all $%
x\in \left[ 0,1\right] $ the inequality 
\begin{equation}
\psi _{h}\left( x\right) =\frac{x\varphi \left( x-h\right) }{c_{h}\left(
\varphi \right) \left( x\right) }<\psi _{h}\left( x+\delta \right) =\frac{%
\left( x+\delta \right) \varphi \left( x+\delta -h\right) }{c_{h}\left(
\varphi \right) \left( x+\delta \right) }  \label{ineqklein}
\end{equation}
holds where 
\begin{equation*}
c_{h}\left( \varphi \right) \left( x\right) :=x\varphi \left( x-h\right)
+\left( 1-x+h\right) \varphi \left( x\right) .
\end{equation*}
Note that $c_{h}\left( \varphi \right) \left( x\right) $ converges to $%
\varphi \left( x\right) $ uniformly in $x$ when $h$ tends to zero.
Inequality (\ref{ineqklein}) means that 
\begin{equation*}
D\left( x,h,\delta \right) :=x\varphi \left( x-h\right) c_{h}\left( \varphi
\right) \left( x+\delta \right) -x\varphi \left( x+\delta -h\right)
c_{h}\left( \varphi \right) \left( x\right) 
\end{equation*}
satisfies the inequality 
\begin{equation}
D\left( x,h,\delta \right) <\delta \varphi \left( x+\delta -h\right)
c_{h}\left( \varphi \right) \left( x\right) .  \label{eqDD}
\end{equation}
By inserting and subtracting $x\varphi \left( x-h\right) c_{h}\left( \varphi
\right) \left( x\right) $ we conclude that 
\begin{eqnarray*}
\frac{D\left( x,h,\delta \right) }{\delta } &=&x\varphi \left( x-h\right) 
\frac{c_{h}\left( \varphi \right) \left( x+\delta \right) -c_{h}\left(
\varphi \right) \left( x\right) }{\delta } \\
&&+xc_{h}\left( \varphi \right) \left( x\right) \frac{\varphi \left(
x-h\right) -\varphi \left( x+\delta -h\right) }{\delta }.
\end{eqnarray*}
If $\varphi \in C^{1}\left[ 0,1\right] $ we can find $\xi _{x,h,\delta }\in %
\left[ x-h,x-h+\delta \right] $ and $\eta _{x,h,\delta }\in \left[
x,x+\delta \right] $ with
\begin{eqnarray*}
\varphi \left( x-h\right) -\varphi \left( x+\delta -h\right)  &=&\varphi
^{\prime }\left( \xi _{x,h,\delta }\right) \cdot \delta  \\
c_{h}\left( \varphi \right) \left( x+\delta \right) -c_{h}\left( \varphi
\right) \left( x\right)  &=&c_{h}\left( \varphi \right) ^{\prime }\left(
\eta _{x,h,\delta }\right) \cdot \delta .
\end{eqnarray*}
It follows that 
\begin{equation*}
\frac{D\left( x,h,\delta \right) }{\delta }=x\varphi \left( x-h\right)
c_{h}\left( \varphi \right) ^{\prime }\left( \eta _{x,h,\delta }\right)
-x\cdot c_{h}\left( \varphi \right) \left( x\right) \varphi ^{\prime }\left(
\xi _{x,h,\delta }\right) .
\end{equation*}
In order to show (\ref{eqDD}) we note that $c_{h}\left( \varphi \right)
\left( x\right) $ converges to $\varphi \left( x\right) ,$ and $c_{h}\left(
\varphi \right) ^{\prime }\left( x\right) $ converges to $\varphi ^{\prime
}\left( x\right) $ for $h\rightarrow 0.$ Hence $D\left( x,h,\delta \right)
/\delta $ converges to $0$ for $h\rightarrow 0$ and $\delta \rightarrow 0,$
and (\ref{eqDD}) holds since   
\begin{equation*}
\frac{D\left( x,h,\delta \right) }{\delta }<\frac{1}{2}m^{2}\leq \frac{m}{2}%
\varphi \left( x\right) ^{2}\leq \varphi \left( x+\delta -h\right)
c_{h}\left( \varphi \right) \left( x\right) 
\end{equation*}
for $m:=\min_{x\in \left[ 0,1\right] }\varphi \left( x\right) $ and $h$
sufficiently small.

Now assume that $\varphi $ is only Lipschitz continuous. Clearly $%
c_{h}\left( \varphi \right) $ is Lipschitz continuous and there exist $M>0$
and $N>0$ such that
\begin{eqnarray*}
\left\vert \varphi \left( x-h\right) -\varphi \left( x+\delta -h\right)
\right\vert &\leq &M\delta \\
\left\vert c_{h}\left( \varphi \right) \left( x+\delta \right) -c_{h}\left(
\varphi \right) \left( x\right) \right\vert &\leq &N\delta
\end{eqnarray*}
where $N$ does not depend on $h.$ It follows that $\left\vert D\left(
x,h,\delta \right) \right\vert /\delta $ is bounded for all $x\in \left[ 0,1%
\right] $ and $h>0$ and $\delta >0.$ If we replace now $\varphi $ by $%
a+\varphi $ we see that 
\begin{eqnarray*}
c_{h}\left( a+\varphi \right) \left( x\right) &=&x\left[ a+\varphi \left(
x-h\right) \right] +\left( 1-x+h\right) \left[ a+\varphi \left( x\right) %
\right] \\
&=&a\left( 1+h\right) +c_{h}\left( \varphi \right) \left( x\right) .
\end{eqnarray*}
Then 
\begin{eqnarray*}
D\left( x,h,\delta ,a+\varphi \right) &=&x\left( a+\varphi \left( x-h\right)
\right) \cdot \left( a\left( 1+h\right) +c_{h}\left( \varphi \right) \right)
\left( x+\delta \right) \\
&&-x\left( a+\varphi \left( x+\delta -h\right) \right) \left( a\left(
1+h\right) +c_{h}\left( \varphi \right) \right) \left( x\right)
\end{eqnarray*}
can be simplified to 
\begin{eqnarray*}
D\left( x,h,\delta ,a+\varphi \right) &=&D\left( x,h,\delta ,\varphi \right)
+ax\left[ c_{h}\left( \varphi \right) \left( x+\delta \right) -c_{h}\left(
\varphi \right) \left( x\right) \right] \\
&&+xa\left( 1+h\right) \left[ \varphi \left( x-h\right) -\varphi \left(
x+\delta -h\right) \right] .
\end{eqnarray*}
On the other hand 
\begin{equation*}
\left( a+\varphi \left( x+\delta -h\right) \right) c_{h}\left( a+\varphi
\right) \left( x\right) \geq a^{2}\left( 1+h\right)
\end{equation*}
and by taking $a>0$ sufficiently large we obtain the desired inequality.
\end{proof}

We shall give an example of a positive function $\varphi \in C\left[ 0,1%
\right] $ such that $Q_{2n}\left( x\right) =B_{2n}\varphi $ does not satisfy
property (W):

\begin{example}
\label{ThmExamp}Let $\varphi _{a}\left( x\right) =a+\left| x-\frac{1}{2}%
\right| $ for $a>0.$ Then 
\begin{equation*}
Q_{2n,a}\left( x\right) :=B_{2n}\varphi _{a}\left( x\right)
=\sum_{k=0}^{2n}\left( a+\left| \frac{k}{2n}-\frac{1}{2}\right| \right) 
\binom{2n}{k}x^{k}\left( 1-x\right) ^{2n-k}
\end{equation*}
has strictly positive Bernstein coefficients, and it satisfies property (W)
if and only if $a>\frac{1}{2}.$
\end{example}

\begin{proof}
It follows that 
\begin{equation*}
\gamma _{2n,k}=\left\{ 
\begin{array}{ccc}
\left( a+\frac{1}{2}-\frac{k}{2n}\right) \binom{2n}{k} & \text{for} & k\leq n
\\ 
\left( a-\frac{1}{2}+\frac{k}{2n}\right) \binom{2n}{k} & \text{for } & 
n<k\leq 2n.
\end{array}
\right. 
\end{equation*}
It follows that 
\begin{equation*}
\frac{\gamma _{2n-1,n-1}}{\gamma _{2n-1,n}}=\frac{na+\frac{1}{2}}{\left(
n+1\right) a}\text{ and }\frac{\gamma _{2n-1,n}}{\gamma _{2n-1,n+1}}=\frac{%
\left( n+1\right) a}{na+\frac{1}{2}}.
\end{equation*}
If $\frac{\gamma _{2n-1,k-1}}{\gamma _{2n-1,k}}$ is increasing then
necessarily $\frac{\gamma _{2n-1,n-1}}{\gamma _{2n-1,n}}<1$ and this implies
that $na+\frac{1}{2}<\left( n+1\right) a,$ which means that $\frac{1}{2}<a.$
Conversely, this condition implies that $\frac{\gamma _{2n-1,n-1}}{\gamma
_{2n-1,n}}<\frac{\gamma _{2n-1,n}}{\gamma _{2n-1,n+1}}$. It is not difficult
to see that the coefficients are increasing.
\end{proof}

Next we want to show by example that the positive polynomials $Q_{n-1}\left(
x\right) $ may not converge in general to a \emph{continuous} function even
if the Bernstein operators $R_{n}$ converge to the identity. In particular
there does not exists in this case a continuous function $\varphi $ with $%
Q_{n-1}=B_{n-1}\varphi $ for all $n\in \mathbb{N}.$

\begin{example}
The rational Bernstein operator $R_{n}$ associated to the nodes $x_{n,k}=%
\sqrt{\frac{k}{n}}$ for $k=0,...,n$ converges to the identity operator but
the associated polynomials $Q_{n-1}\left( x\right) $ defined by 
\begin{equation*}
Q_{n-1}\left( x\right) =\left( 1-x\right) ^{n}+\sum_{k=1}^{n-1}\binom{n-1}{k}%
\prod\limits_{l=1}^{k}\frac{\sqrt{\frac{l}{n}}}{1+\sqrt{\frac{l}{n}}}%
x^{k}\left( 1-x\right) ^{n-1-k}
\end{equation*}
do not converge to a continuous function, in particular $Q_{n-1}$ is not
equal to $B_{n-1}\varphi $ for some continuous function $\varphi \in C\left[
0,1\right] .$
\end{example}

\begin{proof}
Clearly $1/\sqrt{n}\leq \left| x_{n,1}-x_{n,0}\right| \leq \Delta _{n}$ and 
\begin{equation*}
\left| x_{n,k+1}-x_{n,k}\right| =\sqrt{\frac{k+1}{n}}-\sqrt{\frac{k}{n}}=%
\frac{\frac{k+1}{n}-\frac{k}{n}}{\sqrt{\frac{k+1}{n}}+\sqrt{\frac{k}{n}}}%
\leq \frac{1}{\sqrt{n}}\frac{1}{\sqrt{k+1}}.
\end{equation*}
Next we consider for $l=1,..,n-1$ 
\begin{equation*}
\frac{1-x_{n,l}}{x_{n,l}}=\frac{1-x_{n,l}^{2}}{x_{n,l}\left(
1+x_{n,l}\right) }=\frac{1-\frac{l}{n}}{\sqrt{\frac{l}{n}}}\frac{1}{1+\sqrt{%
\frac{l}{n}}}=\frac{n-l}{l}\frac{\sqrt{l}}{\sqrt{n}+\sqrt{l}}.
\end{equation*}
Since $2\sqrt{l}\leq \sqrt{n}+\sqrt{l}$ we can estimate the last factor by $%
1/2.$ It follows that 
\begin{equation*}
\gamma _{n-1,k}=\prod\limits_{l=1}^{k}\frac{1-x_{n,l}}{x_{n,l}}\leq \binom{%
n-1}{k}\frac{1}{2^{k}}
\end{equation*}
and 
\begin{equation*}
Q_{n-1}\left( x\right) \leq \sum_{k=0}^{n-1}\binom{n-1}{k}\frac{1}{2^{k}}%
x^{k}\left( 1-x\right) ^{n-1-k}=\left( 1-\frac{x}{2}\right) ^{n}.
\end{equation*}
Then $Q_{n-1}\left( x\right) $ converges to $0$ for $0<x\leq 1$ but $%
Q_{n-1}\left( 0\right) =1$ for all $n.$ 
\end{proof}

\section{Final Comments}

We want to comment on rational Bernstein operators $R_{n}$ from a different
point of view: Given a strictly positive polynomial $Q_{n-1}\left( x\right) $
we consider the space 
\begin{equation*}
E_{n}=\left\{ \frac{p\left( x\right) }{Q_{n-1}\left( x\right) }:\text{ }%
p\left( x\right) \text{ is a polynomial of degree }\leq n\right\} .
\end{equation*}
Then $E_{n}$ is an extended Chebyshev space over any interval $\left[ a,b%
\right] $, meaning that each non-zero function $f\in E_{n}$ has at most $n$
zeros (including multiplicities) in $\left[ a,b\right] .$ We call a system
of functions $P_{n,k},k=0,...,n$ in an $n+1$ dimensional linear space $E_{n}$
of $C^{n}\left[ a,b\right] $ a Bernstein basis, if each $P_{n,k}$ has
exactly $k$ zeros in $a$ and $n-k$ zeros in $b.$ Thus the system of
functions 
\begin{equation*}
\frac{x^{k}\left( 1-x\right) ^{n-k}}{Q_{n}\left( x\right) },k=0,...,n-1
\end{equation*}
is a Bernstein basis in $E_{n}$ for $\left[ 0,1\right] .$ Bernstein bases in
extended Chebyshev spaces have been studied by many authors, see \cite
{CaPe93}, \cite{CaPe94}, \cite{CaPe96}, \cite{CMP04}, ,\cite{CMP07}, \cite
{Mazu99}, \cite{Mazu05}.

Recently, Bernstein operators for an extended Chebyshev space $E_{n}$ of
dimension $n+1$ have been introduced by J. M. Aldaz, O. Kounchev and the
author which by definition are operators of the form 
\begin{equation*}
S_{n}f\left( x\right) =\sum_{k=0}^{n}f\left( x_{n,k}\right) \alpha
_{n,k}p_{n,k}\left( x\right) 
\end{equation*}
where $p_{n,k}\left( x\right) ,k=0,...,n$, is a Bernstein basis for $E_{n}$.
The nodes $x_{n,k}$ and the weights $\alpha _{n,k}$ are chosen such that $%
S_{n}f_{0}=f_{0}$ and $S_{n}f_{1}=f_{1}$ where $f_{0}$ is a strictly
positive function in $E_{n}$ and $f_{1}\in E_{n}$ has the property that $%
f_{1}/f_{0}$ is strictly increasing. We refer to \cite{AKR09CA}, \cite{AKR09}%
, \cite{AKR10b}, \cite{AKR10} and \cite{Mazu09} for a systematic study
(existence of Bernstein operators fixing two functions and shape preserving
properties) and to \cite{Mazu11} for a discussion of Schoenberg-type
operators in the setting of extended Chebyshev space. It seems to be a
difficult task to establish convergence results of Bernstein operators in
the setting of extended Chebyshev spaces, and the rational Bernstein
operators considered here seems to be the simplest non-trivial example
beyond the classical case of Bernstein operators.


\bigskip

\begin{thebibliography}{99}
\bibitem{AKR09CA}  J. M. Aldaz, O. Kounchev, H. Render, \emph{Bernstein
operators for exponential polynomials,} Constr. Approx. 29 (2009), 345--367.

\bibitem{AKR09}  J. M. Aldaz, O. Kounchev, H. Render, \emph{Shape preserving
properties of generalized Bernstein operators on extended Chebyshev spaces, }%
Numer. Math. 114 (2009), 1--25.

\bibitem{AKR10b}  J. M. Aldaz, O. Kounchev, \emph{Optimality of generalized
Bernstein operators,} J. Approx. Theory 162 (2010), 1407--1416.

\bibitem{AKR10}  J. M. Aldaz, O. Kounchev, H. Render, \emph{Bernstein
operators for extended Chebyshev systems, }Appl. Math. Comput. 217 (2010),
790--800.

\bibitem{KoRe09}  O. Kounchev, H. Render, \emph{On the Bernstein polynomial
of S. Morigi and M. Neamtu, }Result. Math. 53 (2009), 311--322.

\bibitem{CaPe93}  J.-M. Carnicer, J.-M. Pe\~{n}a, \emph{Shape preserving
representations and optimality of the Bernstein basis,} Adv. Comput. Math. 1
(1993), 173--196.

\bibitem{CaPe94}  J.-M. Carnicer, J.-M. Pe\~{n}a, \emph{Totally positive
bases for shape preserving curve design and optimality of B-splines,}
Comput. Aided Geom. Design 11 (1994), 633--654. (1993), 173--196.

\bibitem{CaPe96}  J.-M. Carnicer, J.-M. Pe\~{n}a, \emph{Total positivity and
optimal bases.} In: Total Positivity and its Applications (M. Gasca, C.A.
Micchelli, eds.). Dordrecht: Kluwer Academic, 133--145.

\bibitem{CMP04}  J.-M. Carnicer, E. Mainar, J.M. Pe\~{n}a, \emph{Critical
Length for Design Purposes and Extended Chebyshev Spaces,} Constr. Approx.
20 (2004), 55--71.

\bibitem{CMP07}  J.M. Carnicer, E. Mainar, J.M. Pe\~{n}a, \emph{Shape
preservation regions for six-dimensional space,} Adv. Comput. Math. 26
(2007), 121--136.

\bibitem{Davis}  P. J. Davis, \emph{Interpolation and Approximation,} Dover
Publications, New York 1975.

\bibitem{GaIv12}  I. Gavrea, M. Ivan, \emph{An answer to a conjecture on
Bernstein operators,} J. Math. Anal. Appl. 390 (2012), 86--92.

\bibitem{GoTa09}  H. Gonska, G. Tachev, \emph{A quantitative variant of
Voronoskaja's theorem}, Result. Math. 53 (2009), 287--294.

\bibitem{Lore86}  G.G. Lorentz, \emph{Bernstein polynomials,} Chelsea
Publishing Company, New York 1986 (2nd edition).

\bibitem{MPS01}  E. Mainar, J.M. Pe\~{n}a, J. S\'{a}nchez-Reyes, \emph{Shape
preserving alternatives to the rational B\'{e}zier model,} Comput. Aided
Geom. Design 18 (2001), 37--60.

\bibitem{Mame62}  R.G. Mamedov, \emph{On the asymptotic value of the
approximation of repeatedly differentiable functions by positive linear
operators,} Dokl. Akad. Nauk. 146 (1962), 1013--1016 (in Russian);
Translated in Soviet Math. Dokl. 3 (1962), 1435--1439.

\bibitem{Mazu99}  M.-L. Mazure, \emph{Bernstein bases in M\"{u}ntz spaces,}
Numer. Alg. 22 (1999), 285--304.

\bibitem{Mazu05}  M.-L. Mazure, \emph{Chebyshev Spaces and Bernstein bases,}
Constr. Approx. 22 (2005), 347--363.

\bibitem{Mazu09}  M.-L. Mazure, \emph{Bernstein-type operators in Chebyshev
spaces,} Numer. Alg. 52 (2009), 1016-1027.

\bibitem{Mazu11}  M.-L. Mazure, \emph{Chebyshev-Schoenberg operators, }
Constr. Approx. 34 (2011), 181--208.

\bibitem{PiSa09}  P. Pi\c{t}ul, P. Sablonni\`{e}re, \emph{A family of
univariate rational Bernstein operators,} J. Approx. Theory 160 (2009),
39--55.

\bibitem{Tach12}  G.T. Tachev, \emph{The complete asymptotic expansion for
Bernstein operators,} J. Math. Anal. Appl. 385 (2012), 1179--1183.
\end{thebibliography}
\end{document}